\documentclass[12pt]{amsart}
\usepackage{amsmath,amsfonts,amsthm,amssymb,amscd,url}
\usepackage{dsfont}

\usepackage{bigints}

\DeclareFontFamily{OT1}{rsfs}{}
\DeclareFontShape{OT1}{rsfs}{n}{it}{<-> rsfs10}{}
\DeclareMathAlphabet{\mathscr}{OT1}{rsfs}{n}{it}

\newtheorem{prop}{Proposition}[section]
\newtheorem{theorem}[prop]{Theorem}

\newtheorem*{main}{Theorem}
\newtheorem{corollary}[prop]{Corollary}
\newtheorem{lemma}[prop]{Lemma}

\newtheorem*{defn*}{Definition}
\newenvironment{Rem}{{\bf Remark.}}{}
\numberwithin{equation}{section}
\author{Farzad Aryan}
\title{On an extension of the Landau-Gonek formula}
\begin{document}
\maketitle
\begin{abstract}
We prove an extension of the Landau-Gonek formula. As an application we recover unconditionally some of the consequences of a pair correlation estimate that previously was known under the Riemann hypothesis. As one corollary we prove that at least two-thirds of the zeros of the zeta function are simple under a zero density hypothesis, which is  weaker than the Riemann hypothesis. The results in this paper can be viewed as pair correlation estimates independent of the Riemann hypothesis.
\end{abstract}
\section{\bf Introduction.}
\vspace{0.5 cm}
For $\textbf{Re}(s)>1,$ the Riemann zeta function is defined as
\begin{equation*}
\zeta(s)=\sum_{n=1}^{\infty}\frac{1}{n^s}.
\end{equation*}
It has an analytic continuation to the whole complex plane except $s=1$ and non-trivial zeros in $0<\textbf{Re}(s)<1,$ known as the critical strip. The location of the zeros of the zeta function is closely tied to the distribution of prime numbers. The closer the zeros are to the critical line ($\textbf{Re}(s)=\tfrac{1}{2}$), the better  the control is over the distribution of the prime numbers. The Riemann hypothesis predicts that all of the non-trivial zeros of the Riemann zeta function are located on this line. 
\\

One of the ways the we can see how the prime numbers are connected to the zeros of zeta function is to look at the Landau-Gonek formula. \\

In 1911 Landau~\cite{Lan} proved that for fixed $x>1$
\begin{equation*}
\sum_{\substack{\zeta(\rho)=0\\ 0<\operatorname{Im}(\rho)<T}}{x^{\rho}}= -\frac{T}{2\pi}\Lambda(x)+ O(\log T).
\end{equation*}
Later Gonek~\cite{Gonek} proved a uniform version of Landau's formula: \\ 

Let $x, T>1,$ and for simplicity consider $x$ to be an integer. Then
\begin{equation}
\label{Gonek}
\sum_{\substack{\zeta(\rho)=0\\ 0<\operatorname{Im}(\rho)<T}}{x^{\rho}}= -\frac{T}{2\pi}\Lambda(x)+ O\big(x\log 2xT \log\log 3x\big).
\end{equation}

Gonek's version shows that one can test whether or not $x$ is a prime number by using an averaging sum involving the zeros of the zeta function. The length of the sum should be around the size of $x$. The formula has applications on the statistical distribution of the zeros of the Riemann zeta function. As one application Gonek gave another proof for the following result of Montgomery \cite{Mo}.
\begin{main}[Montgomery] Assume the Riemann hypothesis and fix $0<a<1.$ Then
\begin{align}
\label{mont-gon}
\sum_{0<\gamma, \gamma^{\prime}<T} \bigg(\frac{\sin(\tfrac{\alpha}{2}(\gamma-\gamma^{\prime})\log T)}{\tfrac{\alpha}{2}(\gamma-\gamma^{\prime})\log T}\bigg)^2 \sim \big(\frac{1}{\alpha}+ \frac{\alpha}{3}\big)\frac{T}{2\pi} \log T.
\end{align}
\end{main}
In Theorem \ref{Thm2} we will prove the above result unconditionally. Montgomery proved this result by using his work on the pair correlation of the zeros of the Riemann zeta function. \\


To explain how the Landau-Gonek formula is being applied generally, note that often time working with zeros of $L$-function one need to estimate a sum of the type
\begin{equation}
\sum_{\substack{L(\rho)=0 \\ 0< \operatorname{Im}(\rho) <T}} A(\rho)B(1-\rho),
\end{equation}
where $A(s)= \sum a(n)n^{-s}$ and $B(s)= \sum b(n)n^{-s}$  are Dirichlet polynomials of certain length. Estimating the above comes down to calculating $$\sum_{\substack{L(\rho)=0 \\ 0< \operatorname{Im}(\rho) <T}} \Big(\frac{m}{n}\Big)^{\rho},$$
which is the premise of the Landua-Gonek formula. \\

Here, we try to extend the connections between the zeros of the Riemann zeta function and prime  numbers. To explain the result we assume the Riemann hypothesis, although our results are independent of it. Consider a function $C_{\beta}(\rho)$ that counts the number of zeros of the Riemann zeta function that are close (within the radius $\beta$) of $\rho.$
$$C_{\beta}(\rho)=\# \{\rho^{\prime} : |\rho-\rho^{\prime}|\ll \beta\}.$$ Now imagine that we add this function to the Landau-Gonek formula, in a dot product manner  $\langle C_{\beta}(\rho), x^\rho\rangle$, to have
\begin{equation}
\label{new-l.g}
\sum_{\substack{\zeta(\rho)=0\\ 0<\operatorname{Im}(\rho)<T}}C_{\beta}(\rho){x^{\rho}}.
\end{equation}
We show that for $\beta \gg 1$ the above sum is basically the same as the Landau-Gonek formula, however changes start to emerge for $\beta<1$. For small $\beta,$ adding the factor $C_{\beta}$  in the Landau-Gonek formula would extend its support from primes to products of two primes. 

\begin{theorem}
\label{Thm} Let $M=\alpha \sqrt{\log T}$. Also let $\omega$ be a Gaussian cut-off weight centred around $T$ defined as $$\omega(s)=\frac{1}{\sqrt{\pi}\Delta}\hspace{1 mm}\displaystyle{ e^{\frac{(s-({1}/{2}+iT
))^2}{\Delta^2}}},$$ with $\Delta= T (\log T)^{-1}$. Let $x=r/s$ where $r, s$ are integers with $r,s <T^{1-\epsilon}$. For $x>1$ and $\alpha< \epsilon/2,$ we have that
\small
\begin{align*}
 \notag
2  \pi^{3/2} M  \sum_{\zeta(\rho)=0} & x^{\rho}  \sum_{\zeta(\rho^{\prime})=0}   \omega(\rho^{\prime})e^{{M^2(\rho^{\prime}- \rho)^2}} + O\Big(\frac{1}{T}\Big)   \\ & =
\begin{cases}
\log p\log q\big(e^{-\frac{\log^2(p^i)}{4M^2}}+e^{-\frac{\log^2(q^j)}{4M^2}}\big)  &\text{if $x=p^i q^j$},\\
\\
-\log^2 p\bigg(\displaystyle{ \big(1+ e^{-\frac{\log^2 m}{4M^2}}\big)\frac{\log T}{\log p}+ O(\frac{1}{p})\bigg)} &\text{if $x=p^i$},\\
\\
0 &\text{otherwise.}
\end{cases}
\end{align*} 
\normalsize
For the case $x=1,$ we can take $M< \sqrt{\log T}$ and we have
\begin{align}
\label{eqthm}
 \notag 2 \pi^{3/2}M \sum_{\zeta(\rho=0} \omega(\rho)&\sum_{\zeta(\rho^{\prime})=0}     e^{{M^2(\rho^{\prime}- \rho)^2}}   =2 \sum_{n=1}^{\infty} \frac{\Lambda^2(n)}{n}e^{-\frac{\log^2 n}{4M^2}} \\ &+ \sum_{\zeta(\rho)=0}\omega(\rho)\log(\gamma_{\rho})+O(1).
\end{align}
\end{theorem}

\begin{Rem}
In the LHS of \eqref{eqthm} if we assume the Riemann hypothesis $\rho^{\prime}- \rho$ has no real part and therefore $e^{{M^2(\rho^{\prime}- \rho)^2}}<1$. On the other hand if we have a zero off the critical line this zero will have a large positive contribution to the LHS of \eqref{eqthm} and this contribution can only be affected by zeros that are very close to the zero off the critical line. Now since the RHS of \eqref{eqthm} is bounded by $O(\log T)^2$ the theorem suggest that we cannot have many violation of the Riemann hypothesis. This gives us a weaker version of the Selberg zero density theorem which we put in Corollary \ref{corr}.
\end{Rem}
\\

\begin{Rem}
We can have a similar result with considering two different $L$-functions. For example the sum over $\rho$ can run among the  zeros of the Riemann zeta function and the sum over $\rho^{\prime}$ can be considered to run among the zeros of a Dirichlet $L$-function. \\ 
\end{Rem}

In terms of distribution of imaginary parts of zeros of the zeta function on the torus $\mathbb{R}/\mathbb{Z}$, the Landau-Gonek formula implies that the sequence $\{\alpha \gamma\}$ is uniformly distributed. However when $\alpha$ is very close to $k \log p$ we see some accumulation of points around $-1.$ Ford, Soundararajan, and Zaharescu in~\cite{FZ, FKZ} made this notion more precise. It is interesting that adding an statistical factor to the LHS of the Landau-Gonek formula would result in an arithmetic changes in the RHS. However it is not clear that what it would say regarding the distribution of $\{\alpha \gamma\}$ on the torus. \\

Another way of looking at the distribution of the zeros comes from studying the pair correlation between zeros of the zeta function. Montgomery conjecture~\cite{Mo} states that the pair correlation between pairs of zeros of the Riemann zeta function follows the same distribution as the pair correlation between the eigenvalues of a random Hermitian matrices. More precisely assuming the Riemann hypothesis the Montgomery conjecture asserts that for a Schwartz function $\phi$
\begin{equation}
 \sum_{\substack{0< \gamma, \gamma^\prime< T \\ 0< \gamma- \gamma^\prime < \frac{2\pi \alpha}{\log T}}} \phi\big((\gamma- \gamma^\prime)\frac{\log T}{2\pi}\big) \sim \frac{T}{2\pi}\log T \int_{0}^{\alpha} \phi(x)\bigg(1- \big(\frac{\sin(\pi x)}{\pi x}\big)^2 \bigg)dx.
\end{equation} \\
Regarding the zero free region for the zeros of the Riemann zeta function, by using the Euler product and the functional equation for the zeta function one can show that the zeta function has no non-trivial zeros outside of $0\leq \textbf{Re}(s) \leq 1,$ known as the critical strip. By the work Hadamard and  de la Vall\'{e}e-Poussin we can show that there is no zeros on the line $\textbf{Re}(s)=1$ and do slightly better to show the following zero free region
\begin{equation}
\label{zero-free}
\{\zeta(s) \neq 0 \text{ for } \sigma >1-\frac{c}{\log(|t|+2)} \}.
\end{equation}   Beside improvements on the power of $\log t$, the zero free region \eqref{zero-free} is the best we can get with current method. \\

A natural question is if there exist zeros off the critical line, what would they look like? We show that if these counter examples to the Riemann hypothesis are not too close to each other, we cannot have too many of them.

\begin{corollary}
\label{corr}
For $\epsilon>0,$ let $N$ be the number of zeros of the Riemann zeta function that are distant $(\log T)^{-{1}/{2}+\epsilon}$ apart from the critical line. Moreover, assume that either the distance between imaginary parts of these zeros are bigger than $(\log T)^{-{1}/{2}+\epsilon}$ or the distance between the real part of them is bigger than $(\log T)^{-1+\epsilon}.$ Then we have that $$N\ll \frac{T}{e^{(\log T)^{\epsilon}}}.$$
\end{corollary}
\begin{Rem}
Assuming that zeros do not cluster is somehow a troublesome condition since this is one of the difficulties in dealing with the zeros of the Riemann zeta function. However, we put this corollary here to show a quick application of the first theorem for getting some zero density estimates.\\

\end{Rem}

Next we will prove another theorem, that allow us to look more closely at the microscopic behaviour of the zeros of the Riemann zeta function.
\subsection{\bf Fejer Kernel} In this section we prove a similar result to the Theorem~\ref{Thm} with a different weight and we consider its applications.
\begin{theorem}
\label{Thm2}
For $\alpha<1,$ let
\begin{equation}
\label{def-W}
W_{\rho}(s):= \Big(\frac{T^{\alpha\tfrac{s-\rho}{2}}-T^{-\alpha\tfrac{s-\rho}{2}}}{s-\rho}\Big)^2
\end{equation}
Let $x=r/s$ where $r, s$ are integers with $r,s <T^{1-\alpha}{\log^{-5} T}$. For $x>1$ we have that
\small
\begin{align*}
 \notag
 \sum_{\zeta(\rho)=0} & \omega(\rho)x^{\rho}  \sum_{\zeta(\rho^{\prime})=0}   W_\rho(\rho^{\prime})   + O(\frac{1}{T^{1-\alpha}}) \\ & =
\begin{cases}
\frac{\log p\log q}{2\pi}\Big(\log \big(\frac{T^\alpha}{p^i}\big)\mathds{1}_{p^i< T^\alpha}+ \log\big(\frac{T^\alpha}{p^i}\big)\mathds{1}_{q^i< T^\alpha}\Big)  &\text{if $x=p^i q^j$},\\
\\
-\frac{ \log T\log p}{2\pi}\Big(\alpha \log T-\log\big(\frac{T^\alpha}{p^i}\big)\mathds{1}_{p^i< T^\alpha}\Big) &\text{if $x=p^i$},\\
\\
\frac{\log p\log q}{2\pi q^j}\Big(\log \big(\frac{T^\alpha}{p^i}\big)\mathds{1}_{p^i< T^\alpha}+ \log \big(\frac{T^\alpha}{q^i}\big)\mathds{1}_{q^i< T^\alpha}\Big)  &\text{if $x=p^i q^{-j}$},\\
\\
0 &\text{otherwise.}
\end{cases}
\end{align*}
\normalsize
For the case $x=1,$  we have
\begin{align}
\label{eqthmx=1}
  \sum_{\zeta(\rho)=0}  \omega(\rho)  \sum_{\zeta(\rho^{\prime})=0}   W_\rho(\rho^{\prime})    = \frac{1}{6\pi}(\alpha \log T)^3 + \frac{\alpha}{2\pi} \log^3 T + O(1).\\ \notag
\end{align}
\end{theorem}
By setting $z= (s-\rho)/2,$ we get that $W_\rho(s)$ in \eqref{def-W} equals to $$\Big(\frac{\sin(\alpha \log T z)}{z}\Big)^2,$$  using the complex definition of the Sine function. Therefore we can write \eqref{eqthmx=1} as the following.
\begin{corollary} \label{1--4} Let $0< \alpha<1,$ and consider the Sine function as a complex valued function. We have  that
\begin{align*}
\label{mont-gon}
\sum_{\zeta(\rho)=0 } \omega(\rho) \sum_{\zeta(\rho^{\prime})=0} \bigg(\frac{\sin(\tfrac{\alpha}{2}(\rho-\rho^{\prime})\log T)}{\tfrac{\alpha}{2}(\rho-\rho^{\prime})\log T}\bigg)^2 = \frac{\log T}{2\pi}\big(\frac{1}{\alpha}+ \frac{\alpha}{3}\big) + O\big(\frac{1}{\log^2 T}\big).
\end{align*}
\end{corollary}
Corollary \ref{1--4} is the unconditional version of \eqref{mont-gon}. Rudnick and Sarnak~\cite{RS} also proved pair correlation results unconditionally, assuming the weight satisfy certain conditions including exponential decay. The Fejer Kernel dose not satisfy the exponential decay.  \\ 

A major application of Montgomery's theorem \eqref{mont-gon} was to the problem of simple zeros of the Riemann zeta function. The conjecture is that all of the zeros are simple. Montgomery showed that under the assumption of the Riemann hypothesis at least two-thirds of the zeros  are simple. There were several improvements to this result, the best known conditional lower bound is by by Bui and Heath-Brown ~\cite{BC} that shows that at least $70\%$ of the zeros  are simple . Unconditionally, the best result is due to Bui, Conrey, and Young~\cite{BCY}. They show that more than $41\%$ of the zeros are simple.\\

Similar to  Corollary~\ref{corr}, we can relax the assumption of the Riemann hypothesis and get  Montgomery's result assuming there are not too many  zeros that violate the Riemann hypothesis. We define \begin{equation}
N(\sigma, T) := \# \{\rho : \zeta(\rho)=0 \text{ and } \textbf{Re}(\rho)> \sigma \text{ and } |\textbf{Im}(\rho)|< T \}.
\end{equation}

\begin{corollary}
\label{cor-simple}

For $\sigma> 1/2,$ assume the zero density hypothesis $$ N(\sigma, T) \ll T^{2(1-\sigma)} \log^{-B},$$
with $B>4. $ Then at least two-third of the zeros are simple.
\end{corollary}
We could also get the same corollary by assuming  that the zeros of the Riemann zeta function satisfy the following non-clustering assumption:
\begin{itemize}
\item the real parts of the zeros off the critical line  satisfy $$|\textbf{Re}(\rho)-\tfrac{1}{2}|> \tfrac{\log \log T}{\log T}, $$
\item let $\gamma_1 < \gamma_2 < \cdots < \gamma_m $ be the imaginary parts of zeros off the critical line, then we have
$$ \gamma_{i+1}- \gamma_i> 4.$$
\end{itemize}
It's important to note that this is a somewhat stronger non-clustering assumption that the one we had in corollary \ref{corr}. This comes back to the properties of the Gaussian weight versus the Fejer kernel. \\

The problem of small gap between zeros of the Riemann zeta function is a central problem in number theory due to its connection to the class number formula. Let $\gamma, \gamma^\prime$ denote consecutive ordinates of zeros of the zeta function.  $$\mu := \liminf_{\gamma>0 } \frac{(\gamma- \gamma^\prime)\log \gamma}{2 \pi}.$$
Unconditionally, we do not have much information. Selberg (unpublished, but announced in \cite{Sel1}) showed $\mu<1.$ Conditional to the Riemann hypothesis, many authors~\cite{MO, CGG, BMN, FW, Ps} worked on the problem, and the value of $\mu$ is down to $0.515396.$ Getting $\mu$ smaller than $0.5$ would have significant application in the class number problem. In \cite{Me-Nath} the author and Ng proved that $\mu< 0.49999$ under the assumption of certain (perceived to be hard to prove) conjectural bounds on the auto-correlation of the von Mangoldt and the Liouville functions. Here we prove a result in this direction  under a zero density hypothesis, rather than the Riemann hypothesis.
\begin{corollary}
\label{cor-gap}

For $\sigma> 1/2,$ assume the zero density hypothesis $$ N(\sigma, T) \ll T^{2(1-\sigma)} \log^{-B},$$
with $B>4. $ Then there is a constant $\lambda$ such that the gap between the imaginary parts of the consecutive zeros of the Riemann zeta function is getting smaller than $\lambda$ times the average gap, infinitely often.
\end{corollary}
\begin{Rem}
With a back of the envelope calculation, if we assume $\lambda> 0.78,$  then for any $\gamma,$
$$\sum_{0< \gamma^{\prime}<T} \bigg(\frac{\sin(\tfrac{1}{2}(\gamma-\gamma^{\prime})\log T)}{\tfrac{1}{2}(\gamma-\gamma^{\prime})\log T}\bigg)^2 < 1.33,$$ which is at odds with Corollary \ref{1--4} which implies that $\mu< 0.78.$
\end{Rem}\\

In \cite{Me} we look more closely at the application of the Fejer Kernel in the problem of small gaps between the zeros of the Riemann zeta function.
\section{ \bf Necessary Lemmas and the proof of main formulas}
We begin this section by proving a smooth variant of Landau's formula.
\begin{lemma}
\label{landau}
Let $x>0$ and \begin{equation}
\label{omega-def}
\omega(s)=\frac{1}{\sqrt{\pi}\Delta}\hspace{1 mm}\displaystyle{ e^{\frac{(s-({1}/{2}+iT
))^2}{\Delta^2}}}.
\end{equation}
Then
\begin{align}
\sum_{\zeta(\rho)=0} & {\omega(\rho)}{x^{\rho}}= -\frac{x^{\tfrac{1}{2}+iT}}{2\pi}\sum_{n=1}^{\infty}\frac{\Lambda(n)}{n^{\tfrac{1}{2}+iT}}e^{-{\Delta^2\log^2 \big(\tfrac{x}{n}\big)}} \\ & \notag-\frac{x^{\tfrac{1}{2}+iT}}{2\pi}\sum_{n=1}^{\infty}\frac{\Lambda(n)}{n^{\tfrac{1}{2}-iT}}e^{-{\Delta^2\log^2 ({x}{n})}}-\int_{(1-c)}\frac{\omega(s)}{2\pi i}x^s\frac{\chi'}{\chi}(s)ds.
\end{align}
\end{lemma}
Before giving a proof of Lemma \ref{landau} we state the following lemma from~\cite{Me-Nath}.
\begin{lemma}
\label{winvmt}  Let $c \in \mathbb{R}$ and $x >0$.   Then
\begin{equation}
  \label{winvmt1}
  \frac{1}{2 \pi i} \int_{(c)} \omega(s) x^s ds = \frac{1}{2 \pi} x^{\frac{1}{2}+iT} e^{-\frac{\Delta^2 \log^2 x}{4}},
\end{equation}
\begin{equation}
 \label{winvmt2}
 \frac{1}{2 \pi i} \int_{(c)} \omega(1-s) x^s ds = \frac{1}{2 \pi} x^{\frac{1}{2}-iT} e^{-\frac{\Delta^2 \log^2 x}{4}},
\end{equation}
\begin{equation}
  \label{omegatotalweight}
   \int_{-\infty}^{\infty} \omega(\tfrac{1}{2}+it) dt
=  1.
\end{equation}
\end{lemma}
\begin{proof}[Proof of Lemma \ref{landau}]
By the residue theorem for $c>1$ we have
\begin{align*}
\sum_{\zeta(\rho)=0}{\omega(\rho)}{x^{\rho}}= \frac{1}{2 \pi i} \int_{c-i\infty}^{c+i \infty}  {\omega(s)}{x^s}  \frac{\zeta'}{\zeta}(s) ds- \frac{1}{2 \pi i} \int_{1-c-i\infty}^{1-c+i \infty}  {\omega(s)}{x^s}  \frac{\zeta'}{\zeta}(s) ds
\end{align*}
By using the logarithmic derivative of the functional equation for zeta
\begin{equation}
\label{F.E}
  \frac{\zeta'}{\zeta}(s) = \frac{\chi'}{\chi}(s) - \frac{\zeta'}{\zeta}(1-s).
\end{equation}
we can take the integral on the line $\textbf{Re}(s)=1-c$ back to the line $\textbf{Re}(s)=c$ and since $c>1$ we can use the series expansion of $\frac{\zeta'}{\zeta}$. Therefore by Lemma~\ref{winvmt} we have
\begin{equation*}
  \frac{1}{2 \pi i}\int_{c-i\infty}^{c+i \infty}  {\omega(s)}{x^s}  \frac{\zeta'}{\zeta}(s) ds= -\frac{x^{\tfrac{1}{2}+iT}}{2\pi}\sum_{n=1}^{\infty}\frac{\Lambda(n)}{n^{\tfrac{1}{2}+iT}}e^{-{\Delta^2\log^2 \big(\tfrac{x}{n}\big)}}.
\end{equation*}
For the integral on the line $\textbf{Re}(s)=1-c$ we have
\begin{equation*}
  \frac{1}{2 \pi i}\int_{(1-c)} {\omega(s)}{x^s}  \frac{\zeta'}{\zeta}(s) ds= -\frac{x^{\tfrac{1}{2}+iT}}{2\pi}\sum_{n=1}^{\infty}\frac{\Lambda(n)}{n^{\tfrac{1}{2}-iT}}e^{-{\Delta^2\log^2 ({x}{n})}}-\int_{(1-c)}\frac{\omega(s)}{2\pi i}x^s\frac{\chi'}{\chi}(s)ds.
\end{equation*}
This completes the proof of the lemma.
\end{proof}
Now we prove a similar lemma; the only difference is in the performance of the weight function.
\begin{lemma}
\label{zero-isolation} Let
$\varpi(s)={M}{\pi^{-1/2}}\hspace{1 mm}\displaystyle{ e^{{M^2(s-\rho)^2}}}.$  We have
\begin{align}
\label{eq-lem2}
 \notag\sum_{\zeta(\rho^{\prime})=0}&  x^{\rho^{\prime}}\varpi(\rho^{\prime}) = -\frac{x^{\rho}}{2\pi}\sum_{n=1}^{\infty} \frac{\Lambda(n)}{n^{\rho}}  e^{-\frac{\log^2 (n/x)}{4M^2}} \\ &  - \frac{x^{\rho}}{2\pi}\sum_{n=1}^{\infty}\frac{\Lambda(n)}{n^{1-\rho}}e^{-\frac{\log^2 (nx)}{4M^2}} -   \int_{(1-c)}  \frac{\varpi(s) x^s}{2\pi i} \frac{\chi'}{\chi}(s)ds.
\end{align}
\end{lemma}
\begin{proof}
First of all, note that $\varpi$ is the same as $\omega$ \eqref{omega-def} with $\Delta=M^{-1}.$ Therefore the sum in the LHS of \eqref{eq-lem2} counts zeros that are very close to $\rho.$ By using the residue theorem for $c>1$ we have
\begin{equation}
\label{eq-lem3}
 \notag\sum_{\zeta(\rho^{\prime})=0} x^{\rho^{\prime}}\varpi(\rho^{\prime})= \frac{1}{2 \pi i} \int_{c-i\infty}^{c+i \infty}  {\varpi(s)}{x^s}  \frac{\zeta'}{\zeta}(s) ds- \frac{1}{2 \pi i} \int_{1-c-i\infty}^{1-c+i \infty}  {\varpi(s)}{x^s}  \frac{\zeta'}{\zeta}(s) ds.
\end{equation}
First we expand ${\zeta'}/{\zeta}$ and then we apply Lemma \ref{winvmt} and we have
\begin{equation*}
  \frac{1}{2 \pi i}\int_{c-i\infty}^{c+i \infty} \varpi(s) x^s  \frac{\zeta'}{\zeta}(s) ds= -\frac{1}{2\pi}\sum_{n=1}^{\infty}\Lambda(n) \big(\tfrac{x}{n}\big)^{\rho}e^{-\frac{\log^2 \big(\tfrac{x}{n}\big)}{4M^2}}.
\end{equation*}
For the second integral in \eqref{eq-lem3} first we apply the functional equation \eqref{F.E} and and then the Lemma \ref{winvmt} and we have that it equals to
\begin{equation*}
 -\frac{x^{\rho}}{2\pi} \sum_{n=1}^{\infty} \frac{\Lambda(n)}{n^{1-\rho}} e^{-\frac{\log^2 ({x}{n})}{4M^2}} -\frac{1}{2\pi i} \int_{1-c-i\infty}^{1-c+i \infty} \varpi(s)x^s \frac{\chi'}{\chi}(s)ds.
\end{equation*}
This completes the proof of the lemma.
\end{proof}
By using these lemma we are going to give a proof of our main formula.
\begin{proof}[Proof of Theorem \ref{Thm}] In order to prove the theorem we multiply \eqref{eq-lem2} in Lemma \ref{zero-isolation} with $\omega(\rho)$ and sum over $\rho$ and we also take $\Delta= T (\log T)^{-1}$. Therefore we have three terms to estimate, two term involving the von Mangoldt function, and a term involving the Gamma factor. We begin with
\begin{equation}
\label{eq-I}
\sum_{\zeta(\rho)=0} -\sum_{l=1}^{\infty}\omega(\rho)\big(\frac{r}{ls}\big)^\rho  \Lambda(l) e^{-\frac{\log^2(ls/r)}{4M^2}}
\end{equation}
 We apply Lemma \ref{landau} with $x=r/ls$ and we have \eqref{eq-I} equals to
\begin{align}
\label{eq-II}
\notag \sum_{l=1}^{\infty}  \Lambda(l) e^{-\frac{\log^2(ls/r)}{4M^2}} & \big(\frac{r}{ls}\big)^{1/2+iT} \bigg(\sum_{n=1}^{\infty}  \frac{\Lambda(n)}{n^{1/2+ iT}}e^{-{\Delta^2\log^2(r/lsn)}}+\frac{\Lambda(n)}{n^{1/2- iT}}e^{-{\Delta^2\log^2(rn/ls)}}\bigg) \\ &  + \sum_{l=1}^{\infty}  \Lambda(l) e^{-\frac{\log^2(ls/r)}{4M^2}} \int_{}\omega(\tfrac{1}{2}+it)(r/ls)^{\tfrac{1}{2}+it}\frac{\chi'}{\chi}(\tfrac{1}{2}+it)dt.
\end{align}
Since we considered $r, s< T^{1-\epsilon}$ we have that if $r \neq lsn$ then $$|\frac{r}{lsn} -1|> \frac{1}{T^{1-\epsilon}},$$ and therefore the first term inside the parenthesis in \eqref{eq-II} is very small unless $r/s=ln$ which in that case we get $$\Lambda(l)\Lambda(n)\big(e^{-\frac{\log^2(n)}{4M^2}}+e^{-\frac{\log^2(l)}{4M^2}}\big) .$$ We can use the same type of argument for the term involving the gamma factor to show it is very small unless $r/s=l$ for which we get $-\log T \Lambda(r/s)$.\\

 We need to work out the second term inside the parenthesis in \eqref{eq-II} since we get some off-diagonal contribution in this case. First we separate the diagonal case $rn=ls$ which gives $$\frac{\Lambda(l)\Lambda(n)e^{-\frac{\log^2(n)}{4M^2}}}{n}.$$ Now to get an off-diagonal contribution we need to consider $r,s,l, n$ such that $rn, ls \gg \log^2T/T:= T_0$, otherwise their contributions are very small. Since $r$ is fixed, consider $n$ such that $rn> T_0 $. Therefore $l$ can effectively take values in $$\big[\frac{rn}{s}- \frac{rn}{sT_0},\frac{rn}{s}+ \frac{rn}{sT_0} \big].$$ If $rn<\tfrac{1}{2}sT_0$, then $l$ just can have one value equal to $rn/s.$ In this case we have that the sum
\begin{equation*}
\log T \displaystyle{\sum_{T_0/r< n< sT_0/2r} \frac{\Lambda(n)}{n}e^{-\frac{\log^2 n}{4m^2}}} \ll M\log T \displaystyle{\bigintss_{\log(T_0/r)/2M}^{\log(sT_0/r)/2M}e^{-x^2}dx}.
\end{equation*}
 Since $r< T^{1-\epsilon}$ we have that this error term is smaller than $T^{-\tfrac{\epsilon}{2\alpha}} \log^2 T.$ Recall that $M=\alpha \sqrt{\log T}$. If $rn \geq \tfrac{1}{2}sT_0$ then $l$ can effectively take values in $[\tfrac{rn}{s}- \tfrac{rn}{sT_0}, \tfrac{rn}{s}+ \tfrac{rn}{sT_0}].$ Therefore the sum over $l$ is bounded with $\tfrac{\sqrt{rn}}{\sqrt{s}T_0}$ and the whole thing is bounded with $$\frac{r}{sT_0}\sum_{\tfrac{sT_0}{2r}<n} \Lambda(n)e^{-\frac{\log^2 n}{4m^2}} \ll \frac{r}{sT_0} T^{\alpha^2} \log T \bigintss_{\tfrac{\log(sT_0/r)}{2M}}\displaystyle{e^{-(u-\alpha\sqrt{\log T})^2}}dx.$$ Now if we consider $\alpha< \tfrac{\epsilon}{2}$ we have that the off-diagonal contribution is smaller than $T^{-1}.$
The second term we need to consider is
\begin{equation}
\label{eq-I1}
\sum_{\zeta(\rho)=0} -\sum_{l=1}^{\infty}\omega(\rho)({rl/s})^\rho  \frac{\Lambda(l)}{l} e^{-\frac{\log^2(rl/s)}{4M^2}}
\end{equation}
By using Lemma \ref{landau} with $x=rl/s$ and we have \eqref{eq-I1} equals to
\begin{align}
\label{eq-II1}
\notag \sum_{l=1}^{\infty}  \frac{\Lambda(l)}{l^{1/2-iT}} e^{-\frac{\log^2(rl/s)}{4M^2}} & ({r/s})^{1/2+iT} \bigg(\sum_{n=1}^{\infty}  \frac{\Lambda(n)}{n^{1/2+ iT}}e^{-{\Delta^2\log^2(rl/sn)}}+\frac{\Lambda(n)}{n^{1/2- iT}}e^{-{\Delta^2\log^2(rln/s)}}\bigg) \\ &  + \sum_{l=1}^{\infty}  \Lambda(l) e^{-\frac{\log^2(ls/r)}{4M^2}} \int_{}\omega(\tfrac{1}{2}+it)(rl/s)^{\tfrac{1}{2}+it}\frac{\chi'}{\chi}(\tfrac{1}{2}+it)dt.
\end{align}
For the first term inside the parenthesis in the above the diagonal case $rl=sn.$ Set $l=p^i$ and $n=q^j$ and consequently we have that the diagonal contribution is  $$\frac{\log p \log q}{p^i} e^{-\frac{\log^2(q^j)}{4M^2}}.$$ The off-diagonal terms,  can be estimated similar to the off-diagonal in \eqref{eq-II}. The second term is very small because of the term $e^{-{\Delta^2\log^2(rln/s)}}.$ Also the term with the gamma factor is small by using integration by parts and considering the facts that $r>s$ and that ${\chi'}/{\chi}(\sigma +it)$ behaves like $-\log (t/2\pi).$
\\
\\
Now what we have left is to estimate an error arising from the gamma factor
\begin{equation}
\frac{1}{2\pi }\sum_{\zeta(\rho)=0} \omega(\rho)\int_{} \frac{M}{\sqrt{\pi}} e^{M^2(1-c+it-\rho)^2} (r/s)^{1-c+it} \frac{\chi'}{\chi}(1-c+it)dt.
\end{equation}
To treat this first we move the integral from the line $\textbf{Re}(s)=1-c$ to $\textbf{Re}(s)=\textbf{Re}(\rho)$ and we get an error of a small size from the singularity at $0.$ Therefore we need to estimate
$$\int e^{-M^2(t-t_\rho)^2} (r/s)^{\sigma_\rho+it} \frac{\chi'}{\chi}(\sigma_\rho+it)dt.$$ If $|t-t_\rho|> \log T$ then the above integral is very small. Therefore practically we should consider the integral for $t \in [t_{\rho} - \log T, t_{\rho} + \log T]$ and by using the definition of $\omega$ we can limit the sum over $\rho$ to the one's that are in $[T-T/ \log^{-1}T, T+T \log^{-1}T]$. We use the following expansion for the gamma factor
\begin{equation}
\label{ass-exp-chi}
\frac{\chi^{\prime}}{\chi}(\sigma_{\rho}+it)=- \log\big(\frac{t}{2\pi}\big)- \frac{i(1-2\sigma_{\rho})}{t}+ O\big(\frac{1}{t^2}\big).
\end{equation}
Now if we replace ${\chi'}/{\chi}(\sigma_{\rho}+it)$ with the above expansion (with a bit of calculation) we get that the integral equals
\begin{equation}
\label{last+1}
-\log (t_\rho/2\pi)\int e^{-M^2(t-t_\rho)^2} (r/s)^{\sigma_\rho+it} + O\bigg(\frac{1}{T}e^{-\frac{\log^2 (r/s)}{M^2}} (r/s)^{\sigma_\rho}\bigg) + O\bigg(\frac{(r/s)^{\sigma_{\rho}}}{T^2}\bigg).
\end{equation}
The first term above by using Lemma \ref{winvmt} equals  $$ -\frac{\sqrt \pi}{M}(r/s)^{\rho} e^{\tfrac{-\log^2 (r/s)}{4M^2}}\log (t_\rho/2\pi)$$ and after considering the sum over $\rho$ we have to estimate $$\sum_{\zeta(\rho)=0} \omega(\rho) x^{\rho} \log (t_\rho/2\pi).$$ We prove a following lemma:
\begin{lemma}
\label{lemm-Land-log} For $1< x <T/\log ^2 T,$ we have that
\begin{equation}
\sum_{\zeta(\rho)=0} \omega(\rho) x^{\rho}\log (t_\rho/2\pi)= -\frac{\Lambda(x)}{2\pi}\int \omega(\tfrac{1}{2}+it) \frac{\chi'}{\chi}(\tfrac{1}{2}+it)dt+ O(\frac{x^{1+\epsilon}}{T^2}).
\end{equation}
\end{lemma}
\begin{proof}
We have for $c>1$
\begin{align}
\notag \sum_{\zeta(\rho)=0}{\omega(\rho)}{x^{\rho}}& \frac{\chi'}{\chi}(\rho)= \frac{1}{2 \pi i} \int_{c-i\infty}^{c+i \infty}  {\omega(s)}{x^s}  \frac{\zeta'}{\zeta}(s) \frac{\chi'}{\chi}(s)ds \\ &- \frac{1}{2 \pi i} \int_{1-c-i\infty}^{1-c+i \infty}  {\omega(s)}{x^s}  \frac{\zeta'}{\zeta}(s)\frac{\chi'}{\chi}(s) ds.
\end{align}
Since $c>1$ we have
\begin{equation}
\notag
\int_{c-i\infty}^{c+i \infty}  \omega(s)x^s \frac{\zeta'}{\zeta}(s) \frac{\chi'}{\chi}(s)ds= -\sum_{n=1}^{\infty} {\Lambda(n)} \int_{c-i\infty}^{c+i \infty}  \omega(s){(x/n)^{s}} \frac{\chi'}{\chi}(s).
\end{equation}
If $n=x$ that is  $$\Lambda(n)\int \omega(\tfrac{1}{2}+it)\frac{\chi'}{\chi}(\tfrac{1}{2}+it)dt.$$
For $k\neq n$ we will use the asymptotic expansion \eqref{ass-exp-chi} and therefore we need to estimate we have
$$\int_{c-i\infty}^{c+i \infty}  \omega(c+it){(x/n)^{c+it}} \bigg(\log(t/2\pi) + O(t^{-2})\bigg) dt.$$ By using integration by parts and the fact that $x<T/\log ^2 T,$ we can show that
$$ \int_{c-i\infty}^{c+i \infty}  \omega(c+it){(k/n)^{c+it}} \log t \textit{ } dt,$$ is very small. By using trivial bound we have
$$\int_{c-i\infty}^{c+i \infty}  \omega(c+it){(x/n)^{c+it}} O(\frac{1}{t^2}) \textit{ }  dt \ll \frac{x^c}{T^2}.$$
For the integral on the line $\textbf{Re}(s)=1-c$ we use \eqref{F.E} and we get it is equal to $$\int_{1-c-i\infty}^{1-c+i \infty}  {\omega(s)}{x^s}  \big( \frac{\chi'}{\chi}(s) - \frac{\zeta'}{\zeta}(1-s)\big)\frac{\chi'}{\chi}(s) ds.$$ In this case we do not have a diagonal term and we can bound it by $$O(\frac{x^{1-c}}{T^2}).$$
\end{proof} Now by using Lemma \ref{lemm-Land-log} we have 
$$\sum_{\zeta(\rho)=0} \omega(\rho) x^{\rho} e^{\tfrac{-\log^2 x}{4M^2}}\log t_\rho = \frac{\log T}{2\pi} \Lambda(x) e^{\tfrac{-\log^2 x}{4M^2}} + O\big(\frac{e^{\tfrac{-\log^2 x}{4M^2}}x^{1+\epsilon}}{T}\big).$$
For the error that arises from the  $O$ term above, considering $x=r/s$ with $r, s<T^{1-\epsilon},$ we get that $x<T^{1-\epsilon}$ and consequently the whole error term is bounded by $T^{-1+\epsilon}.$ \\

To see the proof for the case $x=1,$ i.e. \eqref{eqthm} the argument is the same as before, except the diagonal contribution comes from $n=l$ and the contribution from the Gamma factor gives $$\sum_{\zeta(\rho)=0} \omega(\rho) \log t_\rho + O(1).$$ This completes the proof of theorem.
\end{proof}
Now we give a short proof of the Corollary.
\begin{proof}[Proof of Corollary \ref{corr}] For every zero that is off the narrow strip of length $(\log T)^{-1/2 + \epsilon}$  around the critical line we get at least a contribution of size $T^{\epsilon}$ in the LHS of \eqref{eqthm}. By the non-clustering assumption this contribution cannot be canceled out by the zeros in its close vicinity. Therefore the sum of the contribution of these zeros should be bounded by the RHS of \eqref{eqthm}, that is $\log^2 T.$ This will give the desired bound.

\end{proof}
We prove a lemma similar to Lemma \ref{winvmt} but instead of a Gaussian weight we use a generalization of the Fejer kernel.
\begin{lemma}
\label{Fejer} Let $Q>1$ and $ y>0$ and $\rho= \sigma + i\gamma.$ For
\begin{equation}
W_\rho(s)= \Big( \frac{Q^{\frac{s-\rho}{2}}- Q^{-\frac{s-\rho}{2}}}{s-\rho} \Big)^2
\end{equation}
We have
\begin{align*}
\frac{1}{2\pi i} \int_{(c)} y^{s}W_\rho(s)ds=
\begin{cases}
 y^{\rho} \log(Qy) \hspace{3 mm}  \text{ if } \hspace{2 mm} 1>y>Q^{-1},\\
  y^{\rho} \log(Q/y) \hspace{3 mm} \text{if } \hspace{2 mm} Q> y\geq 1,\\
  0 \hspace{21 mm} \text{otherwise.}
\end{cases}
\end{align*}

We also have
\begin{align*}
\frac{1}{2\pi i} \int_{(c)} y^{s}W_\rho(s)ds=
\begin{cases}
 y^{1-\rho} \log(Qy) \hspace{3 mm}  \text{ if } \hspace{2 mm} 1>y>Q^{-1},\\
  y^{1-\rho} \log(Q/y) \hspace{3 mm} \text{if } \hspace{2 mm} Q> y\geq1,\\
    0 \hspace{21 mm} \text{otherwise.}
\end{cases}
\end{align*}
\end{lemma}

\begin{proof}
We have that
$$W_\rho(s)= \frac{Q^{{s-\rho}}}{(s-\rho)^2}+ \frac{Q^{{\rho-s}}}{(s-\rho)^2}- \frac{2}{(s-\rho)^2}.$$ Therefore for $c \neq \sigma$ we set $s-\rho= z$ and we have
$$\int y^sw(s)= \int y^\rho\frac{(Qy)^{{z}}}{z^2}+ y^\rho\frac{(y/Q)^{z}}{z^2}- y^\rho\frac{2 y^z}{z^2}.$$

To continue we use the follwoing version of the perron's formula
\begin{align*}
\frac{1}{2\pi i} \int_{(c)} x^{z}\frac{ds}{z^2}=
\begin{cases}
 \log x \hspace{3 mm}  \text{ if } \hspace{2 mm} x>1,\\
  0   \hspace{2 mm}  \text{otherwise. }
\end{cases}
\end{align*}
For $1>y> 1/Q$ we get $\log Qy,$ for $Q> y>1$ we get $\log Qy- 2\log y,$ and for $y>Q$ everything will cancel out. This finishes the proof of the first part. For the second part note that
\begin{align*}
\int y^s W_\rho(1-s)ds = \int y^s\Big( \frac{Q^{\frac{s-(1-\rho)}{2}}- Q^{-\frac{s-(1-\rho)}{2}}}{s-(1-\rho)} \Big)^2 ds,
\end{align*}
therefore the proof is the same just we need to switch $\rho$ to $1-\rho.$
\end{proof}
We apply the above lemma with $Q=T^\alpha$ to get the following
\begin{lemma}
\label{lem-Fej} For $x>0,$ we have
\begin{align}
\notag \sum_{\zeta(\rho^{\prime})=0} &   W_{\rho}(\rho^{\prime}) x^{\rho^{\prime}}= -x^\rho \sum_{x \leq n<T^\alpha x} \frac{\Lambda(n)}{n^{\rho}}\log\big(\frac{T^\alpha x}{n}\big)-x^\rho \sum_{n<T^\alpha /x} \frac{\Lambda(n)}{n^{1-\rho}}\log\big(\frac{T^\alpha}{nx}\big) \\ &-x^\rho \sum_{x/T^\alpha< n<x} \frac{\Lambda(n)}{n^{\rho}}\log\big(T^\alpha\frac{n}{x}\big) -\int_{(1-c)} x^s \frac{\chi'}{\chi}(s) W_\rho(s)\frac{ds}{2\pi i} +  W_{\rho}(1) x.\\ \notag
\end{align}
For $x=1$ this simplifies to
\begin{align}
\label{casex=1}
\notag \sum_{\zeta(\rho^{\prime})=0} &   W_{\rho}(\rho^{\prime}) = - \sum_{ n<T^\alpha } \frac{\Lambda(n)}{n^{\rho}}\log\big(\frac{T^\alpha }{n}\big)-  \sum_{ n<T^\alpha } \frac{\Lambda(n)}{n^{1-\rho}}\log\big(\frac{T^\alpha }{n}\big) \\ &  -\int_{(1-c)} \frac{\chi'}{\chi}(s) W_\rho(s)\frac{ds}{2\pi i} +  W_{\rho}(1) .\\ \notag
\end{align}
\end{lemma}
\begin{proof}
The proof goes similarly to the proof of Lemma \ref{landau}. By the residue theorem, for $c>1$ we have
\begin{align*}
\sum_{\zeta(\rho^{\prime})=0} &   W_{\rho}(\rho^{\prime}) x^{\rho^{\prime}}= \frac{1}{2 \pi i} \int_{(c)}  W_{\rho}(s){x^s}  \frac{\zeta'}{\zeta}(s) ds- \frac{1}{2 \pi i} \int_{(1-c)}  W_{\rho}(s){x^s}  \frac{\zeta'}{\zeta}(s) ds.
\end{align*}
For the integral on $\textbf{Re}(s)=c$ we expand the logarithmic derivative of zeta and apply Lemma \ref{Fejer} to get the result. For the integral on $\textbf{Re}(s)=1-c$ we use the functional equation  \eqref{F.E} and then apply Lemma \ref{Fejer}.
\end{proof}

\begin{proof}[Proof of Theorem \ref{Thm2}] We begin by multiplying \eqref{casex=1} in Lemma \ref{lem-Fej}  with $\omega(\rho)x^\rho$ and sum over $\rho.$ Therefore we get
\begin{align*}
 \sum_{\zeta(\rho)=0} &  \omega(\rho)x^{\rho}  \sum_{\zeta(\rho^{\prime})=0}   W_\rho(\rho^{\prime}) =\sum_{n<T^{\alpha }} \Lambda(n)\log \big(\frac{T^{\alpha}}{n}\big)\Big(-\sum \omega(\rho) (\tfrac{x}{n})^{\rho}\Big)\\ &  + \sum_{}W_{\rho, T}(1)\omega(\rho)x^{\rho} + \sum_{n<T^{\alpha}}\frac{ \Lambda(n)}{n}\log \big(\frac{T^{\alpha}}{n}\big) \Big(-\sum \omega(\rho) ({xn})^{\rho}\Big) \\ &
+\sum \omega(\rho)x^\rho \int_{(1-c)}  \frac{\chi'}{\chi}(s) W_\rho(s)\frac{ds}{2\pi i}
\end{align*}
If $x$ equals a prime smaller than $T^\alpha,$ ($x/n=1$ in the first sum) then we get $-(2\pi)^{-1}\Lambda(x)\log T\log(T^\alpha/x).$ \\

Now assume that $x/n \neq 1$ we use the Lemma \ref{landau} on sums over $\rho$, and given $x=r/s$ we have the first term above equal to \\
\begin{align*}
\frac{1}{2\pi} \sum_{n<T^{\alpha }}\Lambda(n)\log \big(\frac{T^{\alpha}}{n}\big)\big(\frac{r}{sn}\big)^{\tfrac{1}{2}+ iT}\sum_m \frac{ \Lambda(m)}{m^{\tfrac{1}{2}+ iT}}e^{-\Delta^2 \log^2(\tfrac{r}{smn})}+ \frac{ \Lambda(m)}{m^{\tfrac{1}{2}- iT}}e^{-\Delta^2 \log^2(\tfrac{rm}{sn})}
\end{align*}
plus a negligible error term coming from the gamma factor in Lemma \ref{landau}. Recall that we took $\Delta= T (\log T)^{-1}$. We are going to show that the contribution from the sum over $m, n$ would be just diagonal, coming from $r=smn$ and $rm=sn.$ To show this, by using the assumption we have $sn< T^\alpha r,$ and since we assumed $r< T^{1- \alpha} \log^{-5}T$ we get $sn<  T \log^{-5}T,$ and because of this the off-diagonal contribution is negligible. By the above argument we get $$\frac{\Lambda(m)\Lambda(n)}{2\pi } \big(\log(\frac{T^\alpha}{m})\mathds{1}_{m<T^\alpha}+ \log(\frac{T^\alpha}{n})\mathds{1}_{n<T^\alpha}\big),$$ for $r/s= mn$ and we get $$\frac{\Lambda(m)\Lambda(n)}{2\pi m}  \log(\frac{T^\alpha}{n})\mathds{1}_{n<T^\alpha},$$ for $r/s= n/m.$ The other terms can be handled similarly. Now we will take care of the Gamma factor: $$\sum_{\rho} \omega(\rho)(\frac{r}{s})^\rho \int_{(1-c)}  \frac{\chi'}{\chi}(z) W_\rho(z)\frac{dz}{2\pi i}.$$
We take the integral to $\textbf{Re}(s)= \textbf{Re}(\rho)$ and we get an error term of size $T^{\alpha \textbf{Re}(\rho)}|\rho|^{-2}.$ Considering the support of $\omega(\rho)$ this is bounded by $T^{-2+ \alpha}. $ Therefore we need to calculate
$$ \int_{(0)}  \frac{\chi'}{\chi}(\rho + z) W_{0}(z)\frac{dz}{2\pi i}. $$ We will use \eqref{ass-exp-chi} for the expansion of the gamma factor along with Lemma~\ref{Fejer} and similar to the proof of Theorem \ref{Thm} we get $$\alpha \log T \sum \omega(\rho) (\tfrac{r}{s})^\rho \log(\tfrac{\gamma_\rho}{2\pi}) + O(\frac{\sum \omega(\rho )x^{\textbf{Re}(\rho)} \log^5 T}{T}).$$
We apply the Lemma \ref{lemm-Land-log} and we have the contribution of the gamma factore is
$$-\alpha \log^2 T \frac{\Lambda(r/s)}{2\pi} + O(\frac{\sum \omega(\rho )x^{\textbf{Re}(\rho)} \log^5 T}{T}).$$
 The error is bounded by $T^{1-\alpha},$ this finishes the proof for $x\neq 1$. For $x=1$  by using \eqref{casex=1} we have
 
\begin{align*}
\notag \sum_{\rho}\sum_{\zeta(\rho^{\prime})=0} &   W_{\rho}(\rho^{\prime}) = - \sum_{ n<T^\alpha } \Lambda(n)\log\big(\frac{T^\alpha }{n}\big) \Big(\sum_{\rho}\frac{\omega(n)}{n^{\rho}}+ \frac{\omega(n)}{n^{1-\rho}}\Big) \\ &  -\sum_{\rho}\int_{(1-c)} \frac{\chi'}{\chi}(s) W_\rho(s)\frac{ds}{2\pi i} + \sum_{\rho} W_{\rho}(1) .\\ \notag
\end{align*}
Now applying the smooth version of the Landau-Gonek formula and the same treatment of the gamma factor as the case for $x\neq 1$ we get that the expression above is equal to
\begin{align*}
\frac{1}{\pi} & \sum_{ n<T^\alpha } \frac{\Lambda^2(n)}{n}\log\big(\frac{T^\alpha }{n}\big)+ \frac{\alpha \log^3 T}{2 \pi} + O(1) \\
& = \frac{1}{6\pi}(\alpha \log T)^3 + \frac{\alpha}{2\pi} \log^3 T + O(1).
\end{align*}
\end{proof}
Next, we will give a proof of Corollaries \ref{cor-simple} and \ref{cor-gap}.
We will prove these corollaries by showing that the effect of the zeros that violate the Riemann hypothesis are negligible in the LHS \eqref{eqthmx=1}. 

\begin{proof}
We assume the density hypothesis in the following form
\begin{equation}
\label{zero-den-hyp}
N(\sigma, T) \ll T^{2(1-\sigma)} \log^{-B}.
\end{equation}
Let us partition the set of zeros that off the critical line to
\begin{align*}
& \rho_i = \{\rho, 1-\rho : \tfrac{1}{2}+\frac{i}{\log T}< \textbf{Re}(\rho) \leq \tfrac{1}{2}+ \frac{i+1}{\log T} \} \\ &
\rho_{i, m}=  \{\rho \in \rho_i : m-1 < \textbf{Im}(\rho) \leq m \}
\end{align*}
Let $\rho= 1/2 + \sigma + i\gamma$ and $\rho^{\prime}= 1/2 - \sigma^{\prime} + i\gamma^{\prime}$ then we have  $$W_{\rho}(\rho^{\prime})= \frac{T^{\sigma+ \sigma^{\prime} + i(\gamma - \gamma^{\prime})}+ T^{-(\sigma+ \sigma^{\prime} + i(\gamma - \gamma^{\prime}))}-2}{(\sigma+ \sigma^\prime + i(\gamma- \gamma^\prime))^2},$$
Therfore if we consider the effect of zeros in $\rho_{i, m}$ and $\rho_{j, m^{\prime}}$ together we get that its bounded by
\begin{align*}
\frac{T^{i+j/\log T}}{1+ (m-m^{\prime})^2} \log^2 T
\end{align*}
For $m-m^{\prime}=h,$ where $m, m^\prime$ are in the effective range of $\omega,$ we have that the total effect is bounded by
\begin{align}
\label{cauchy-sch}
\notag \log^2 T\sum_{i,j} \frac{T^{i+j/\log T}}{h^2} & \sum_{m-m^{\prime}=h}|\rho_{i, m}||\rho_{j, m^\prime}| \\ &  \ll \sum_{i,j} \frac{T^{i+j/\log T}}{h^2} \Big(\sum_{m}|\rho_{i, m}|^2\Big)^{\tfrac{1}{2}} \Big(\sum_{m}|\rho_{j, m}|^2\Big)^{\tfrac{1}{2}}
\end{align}
Using the zero counting formula we have $|\rho_{i, m}|< \log T,$ therefore
$$\sum_{m}|\rho_{i, m}|^2 \ll \log^2 T \sum_m |\rho_{i, m}|< \log^{-B+2} T^{1-\tfrac{2i}{\log T}}.$$
The last inequality is by using the zero density hypothesis \eqref{zero-den-hyp}.
Putting this back into~\eqref{cauchy-sch} and apply it to
\begin{align}
  \sum_{\substack{\zeta(\rho)=0 \\ \textbf{Re}(\rho)\neq 1/2 } }  \omega(\rho)  \sum_{\substack{\zeta(\rho^{\prime})=0 \\ \textbf{Re}(\rho^{\prime})\neq 1/2 } }   W_\rho(\rho^{\prime}) 
\end{align}
  and we get that the total effect of zeros off the critical in the above line is bounded by $$T \log^{-B+7}$$ considering $|\omega(\rho)| \ll T^{-1}\log T. $ Therefore for $B>4$ this gets absorbed in the error term considering the main term in \eqref{eqthmx=1}  has size $\log^3 T.$ The rest of the proof goes the same as the proof of Corollary 2 in \cite{Mo}.
\end{proof}

E-mail address: farzad.aryan@mathematik.uni-goettingen.de

\end{document}